\documentclass{amsart}

\numberwithin{equation}{section}

\usepackage{amssymb}
\usepackage{enumerate, xspace}
\usepackage{dsfont}
\usepackage[colorlinks]{hyperref}

\newtheorem{thmm}{Theorem}
\newtheorem{thm}{Theorem}[section]
\newtheorem{lem}[thm]{Lemma}

\newtheorem{rem}[thm]{Remark}

\newcommand\cC{{\mathcal C}}

\newcommand\cL{{\mathcal L}}
\newcommand\cO{{\mathcal O}}

\newcommand\bC{{\mathbb C}}

\newcommand\bN{{\mathbb N}}

\newcommand\bR{{\mathbb R}}

\newcommand\ve{\varepsilon}

\newcommand\vf{\varphi}

\newcommand\Id{{\mathds{1}}}

\renewcommand{\div}{{\operatorname{div}\,}}
\begin{document}

\title[Expanding maps with singularities]{Multidimensional expanding maps with singularities: a pedestrian approach}
\author{Carlangelo Liverani}
\address{Carlangelo Liverani\\
Dipartimento di Matematica\\
II Universit\`{a} di Roma (Tor Vergata)\\
Via della Ricerca Scientifica, 00133 Roma, Italy.}
\email{{\tt liverani@mat.uniroma2.it}}
\date{October 10, 2011}
\begin{abstract}
I provide a proof of the existence of absolutely continuous invariant measures (and study their statistical properties) for multidimensional piecewise expanding systems with not necessarily bounded derivative or distortion. The proof uses basic properties of multidimensional BV functions.
\end{abstract}
\thanks{It is a pleasure to thank Viviane Baladi and Gabriella Tarantello for helpful discussions. Also I like to thanks the Fields Institute, Toronto, where the paper was completed. Work supported by the European Advanced Grant MALADY (ERC AdG 246953). Last, I'd like to thank the anonymous referee for pointing out a really silly mistake and saving me from public embarrassment.}
\maketitle


\section{introduction}\label{sec:intro}
Lately several authors attempted to treat billiards directly by transfer operator methods. While the first attempts can be traced, at least, as far back as 1995,\footnote{ See \cite{Li1} which was in fact applicable to billiards by introducing the usual homogeneity strips strategy, even though such an extension was not present in the published version.} most of the related work is quite recent and is motivated by the anisotropic Banach space approach started in \cite{BKL}. Such an approach has been very successful in the smooth case \cite{GL1,GL2, Ba, BT1,BT2, Li2, Li3, LT, BL} but its extension to the piecewise smooth case is still unsatisfactory in spite of encouraging recent progress, see \cite{DL, BG1, BG2, BaLi}. 
The hopes motivating such a program are to obtain simpler, stronger and more flexible proofs for, e.g., decay of correlations, statistical stability and limit laws.

Given such a situation, it is natural to reconsider the piecewise expanding case. The payoff of this revisitation is twofold. On the one hand I generalize the existing results to the case in which the derivative of the map is unbounded and there may be countably many domains of smoothness. On the other hand I obtain a more streamlined and elementary proof. Hence this note both contains new results and is a useful warm up for the hyperbolic setting. 

More precisely, I show that the blow up of the derivative can be, at least in some cases, easily handled without resorting to the ``homogeneity strips" strategy usually employed in billiards. The method presented here is the higher dimensional analogous of the one introduced in \cite{Ry}. I work in the piecewise $\cC^2$ setting. One could extend the theory to the $\cC^{1+\alpha}$ case, but she would have to use a generalization of $BV$ spaces that would considerably cloud the argument (but see \cite{Th} for a possible alternative).

I use the space of functions of bounded variation (unlike some more unconventional spaces as in \cite{Sa}) and very little of the theory of bounded variation functions (e.g. no computations based on the traces of a BV function is used, contrary to most of the literature starting from \cite{Ke, GB} to the more recent \cite{Co1, Co2}).

\section{The class of maps}
Let us consider classes of maps $T:\Omega\to\overline\Omega$, $\Omega=\mathring  \Omega\subset [0,1]^d$, satisfying specified subsets of the following conditions.\footnote{ If a given $T$ does not satisfy the conditions below, they may still hold for some power of $T$.}
Let $m$ be the Lebesgue measure.
\begin{enumerate}[{\bf 1.}]  \setcounter{enumi}{-1}
\item \label{c:0} For each $n\in\bN$, $m(T^{-n}\partial\Omega)=0$.
\item \label{c:1}There exists a countable collection of disjoint connected open sets $\{\Omega_k\}$, $\cup_k\Omega_k=\Omega$,\footnote{ Note that, with the present definition, $\Omega$ is necessarily disconnected unless $\{\Omega_k\}$ consists of a single element. Moreover, one can check that  $\overline{\cup_k\partial\Omega_k}=\partial\Omega$.} such that $T$ is $\cC^2$, with $\cC^1$ inverse, when restricted to each $\Omega_k$.

\item \label{c:2} For all $k\in\bN$, $T|_{\Omega_k}$ and $(DT)^{-1}|_{\Omega_k}$ have a continuous extension to $\overline\Omega_k$.\footnote{ I will use the notations $(D_xT)^{-1}$ or $[(DT)^{-1}](x)$ for the inverse of the matrix $DT$ computed at the point $x$.}

\item  \label{c:4} For each $j\in\{1,\dots,d\}$,\footnote{ From now on the notation $x_{\neq j}$ stands for the vector $(x_1,\dots,x_{j-1},x_{j+1},\dots, x_d)$ and $(x_{\neq j},y)$ stands for $(x_1,\dots,x_{j-1},y,x_{j+1},\dots, x_d)$.} $x_{\neq j}\in\bR^{d-1}$ and $k\in\bN$, $\Omega_k\cap \{(x_{\neq j},y)\}_{y\in\bR}=\cup_\ell \{x_{\neq j}\}\times I_{k,\ell}(x_{\neq j})$, where the $I_{k,\ell}(x_{\neq j})\subset \bR$ are, possibly countably many, disjoint open intervals.\footnote{ Note that, for some $x_{\neq j}$, the set $\{I_{k,\ell}(x_{\neq j})\}$ may be empty.} In addition, there exist $\delta>0$ and $\lambda>1$ such that\footnote{ The following could be replaced by a weaker condition in which one takes an inf on all the possible orthonormal bases. This is due to the fact that in definition \eqref{eq:alpha0} appears the divergence which is rotationally invariant. I avoid this generalization to favor readability. In particular, to implement such a strategy I would be forced to use a rotation invariant norm, such as the Euclidean one, instead of $\|\vf(x)\|_\infty=\max_{i\in\{1,\dots,d\}}|\vf_i(x)|$, $\|A(x)\|_\infty=\max_i\sum_j|A(x)_{i,j}|$, $\|\vf\|_{L^\infty}=\sup_{x}\|\vf_i(x)\|_\infty$ and $\|A\|_{L^\infty}=\sup_x\|A(x)\|_\infty$, which I find more convenient here.}
\[
\sup_{x\in [0,1]^d}\sup_j\sum_{\{k,\ell\;:\; [x_j-\delta,x_j+\delta]\cap I_{k,\ell}(x_{\neq j})\neq \emptyset\}}\left\|\Id_{I_{k,\ell}(x_{\neq j})}(\cdot)[(DT)^{-1}](x_{\neq j}, \cdot)\right\|_{L^\infty}\leq\lambda^{-1},
\]
where $\Id_A$ stands for the indicator function of the set $A$.
\item \label{c:5}
For each $k\in\bN$ let $\Omega_k^\ve=\{x\in\Omega_k\;:\; \operatorname{dist}(x,\partial\Omega_k)\geq \ve\}$ and $\partial^\ve\Omega_k=\Omega_k\setminus\Omega^\ve_k$. Then, for each $j\in\{1,\cdots, d\}$,
\[
\lim_{\ve\to 0}\sup_{x_{\neq j}}\sum_{k}\int_0^1\|\partial_{x_j}[(DT)^{-1}](x_{\neq j},y)\|_\infty\Id_{\partial^\ve\Omega_k}(x_{\neq j},y)\, dy=0.
\]
\item\label{c:3} For each $i,j,k\in\{1,\dots,d\}$, $(DT)_{kj}[(DT)^{-1}]_{ji}\in L^1(\Omega,m)$. 
\item \label{c:6}There exist $C>0$, $\alpha\in (0,1]$ and $a\geq1$ such that for each $x\in\Omega$ and $\ve>0$ holds true\footnote{ One could even allow an exponential blow up at singularities, provided the rate is taken small enough. It suffices to truncate the restricted transfer operators (defined in section \ref{sec:open}) at a polynomial distance from the singularity rather than at an exponential distance. I choose not to pursue this generalization because on the one hand it is obvious and on the other hand it would only make the argument less transparent.}
\[
m(\cup_k\partial^\ve\Omega_k)\leq C^d\ve^\alpha\;;\quad \|\nabla\det( D_xT)\|_\infty\leq Cd(x,\partial\Omega)^{-a}.
\]
\end{enumerate}
\begin{rem}The above conditions are not all independent, I introduced some partial repetitions to make them more transparent. The zeroth condition simply ensures that the dynamics is well defined Lebesgue almost surely. In the case $d=1$, the first two conditions reduce to the usual definition of piecewise expanding maps.  Condition three reduces, in one dimension and when the partition is finite, to the standard condition that the expansion be larger than two.\footnote{ In fact, it is slightly weaker insofar it requires that the sum of the inverse of the minimal expansion of any two consecutive intervals is less than one. At any rate in one dimension the larger than two condition can always be satisfied, for some power of $T$, provided the expansion is larger than one. This in not true in higher dimensions (see \cite{Ts1, Bu2}) although it can be achieved in special cases (see \cite{Bu1,Bu3,Ts2}).}  In higher dimensions, for a finite partition, condition three is implied by the condition that the boundaries of the $\Omega_k$ are transversal to the coordinate axes and that the maximal number of $\overline\Omega_k$ covering a point is strictly bounded by the minimal expansion ($\|(DT)^{-1}\|_{L^\infty}^{-1}$). In one dimension condition four simply says that the inverse of the derivative must be locally in $W^{1,1}$ (i.e. its derivative must be in $L^1$), this is essentially Rychlik's condition \cite{Ry}. In higher dimensions, assuming a finite partition with transversal boundaries, condition four is implied by a uniform integrability conditions for the functions $\|\partial_{x_j}[(DT)^{-1}](x_{\neq j},\cdot)\|_\infty$.\footnote{ I doubt that having only the inverse of the derivative locally in $W^{1,1}$ suffices to prove Theorem \ref{thm:main}.} Condition five trivializes in one dimension and is fairly weak in any dimension. Finally, the last condition is not needed if $d=1$ (see Theorem \ref{thm:main}).
\end{rem}
As in the one dimensional case we will work with the $BV$ and the $L^1$ norms, but $BV$ needs a word of explanation.
For each $h\in L^1(\bR^d, m)$  define\footnote{ By $\cC^r_0$ I mean the vector space of $r$-times differentiable functions with compact support. Here and in the following I will not specify the variable of integration if no confusion arises. In any case all the integrations are with respect to the Lebesgue measure unless otherwise stated.}
\begin{equation}\label{eq:alpha0}
\|h\|_{BV(\bR^d)}:=\sup_{\substack{\vf\in\cC^1_0(\bR^d,\bR^d)\\ \|\vf\|_{L^\infty}\leq 1}}\int_{\bR^d} h\cdot\div\vf\,.
\end{equation}
As usual $BV(\bR^d)=\{h\in L^1(\bR^d,m)\;:\; \|h\|_{BV(\bR^d)}<\infty\}$. Next,
\begin{equation}\label{eq:alpha}
\|h\|_{BV(\Omega)}:=\|\Id_\Omega h\|_{BV(\bR^d)},
\end{equation}
and $BV(\Omega)=\{h\in BV(\bR^d)\;:\; \operatorname{supp} h\subset \Omega\}$. From now on I will use $BV$ to designate $BV(\Omega)$, if no confusion arises.\footnote{ Note that $BV(\Omega)$ is a closed vector space under the $BV(\bR^d)$ norm and that \eqref{eq:alpha} is a norm in $BV(\Omega)$. Also, for $h\in L^1(\Omega,m)$ I will often write $h\in BV(\Omega)$ to mean $\Id_\Omega h\in BV(\Omega)$, i.e. the functions are extended to be zero on $\Omega^c$.} For the reader convenience I collect in Appendix \ref{sec:bv} the relevant properties of BV functions.
\begin{rem}
Note that the alternative to define the norm in \eqref{eq:alpha} for each $h\in L^1(\Omega)$ directly by \eqref{eq:alpha0} with the sup restricted to the $\vf$ with support contained in $\Omega$ would not work (in particular Lemma \ref{lem:lasota-yorke} may fail, see \cite{PGB} for a counterexample).
\end{rem}

We are now ready to state precisely the result proved in this note.

\begin{thmm}\label{thm:main}
Any map satisfying assumptions {\bf \ref{c:0}-\ref{c:3}} has (at most finitely many) invariant measures absolutely continuous with respect to Lebesgue. Such measures have densities that belong to $BV$ and their supports yield the ergodic decomposition of the Lebesgue measure.\footnote{ That is, all the indecomposable invariant sets ($\!\!\!\mod 0$) of positive Lebesgue measure.} If there exists only one absolutely continuous invariant measure $\mu$, with density $h_*$, then the following alternative holds: either the dynamical system $(\Omega, T,\mu)$ is not mixing or there exist $\sigma\in(\lambda^{-1},1)$ and $C>0$ such that, for all $h\in BV(\Omega)$ and $f\in L^d(\Omega,m)$,\footnote{ To be precise $f\circ T^n$ is defined only on $\cap_{n\in\bN}T^{-n}\Omega\subset\Omega$ which differs from $\Omega$ for a zero measure set by hypothesis {\bf\ref{c:0}}. Since it is irrelevant, I will not comment on this any further.}
\[
\left|\int_\Omega f\circ T^n h-\int_\Omega h \cdot \int_\Omega h_* f\right|\leq C\|h\|_{BV}\|f\|_{L^d}\; \sigma^{n}.
\]
If the first alternative holds true, then there exist $N\in\bN$ and disjoint  sets $\{\widetilde \Omega_j\}_{j=1}^N$, $\Omega=\cup_j\widetilde \Omega_j$ Lebesgue a.s., such that $T\widetilde\Omega_j=\widetilde\Omega_{j+1}$ for $j<N$ and $T\widetilde\Omega_N=\widetilde\Omega_{1}$. Moreover, for each $j\leq N$, $(T^N, \widetilde \Omega_j,\mu)$ is an exponentially mixing dynamical system.

Finally, if $d=1$ or if the map satisfies also assumption {\bf \ref{c:6}}, then there exists $\ve_*>0$ such that for each absolutely continuous invariant measure the associated density is uniformly positive in an open ball of size $\ve_*$. In particular, the ergodic components of the Lebesgue measure are open ($\!\!\!\mod 0$).
\end{thmm}

\begin{rem} Note that the above Theorem is constructive since $\ve_*$ can be computed by looking at the proof and keeping track of the various constants.\footnote{ I choose not to do this in favor of readability.} This provides a bound on the number of ergodic components. One can then prove that there is only one ergodic component by finding an $\ve_*$-dense orbit. Moreover, if there is only one invariant measure and two periodic orbits (away from $\partial\Omega$) that have coprime periods, then the measure is mixing.
\end{rem}
\begin{rem} It is important to remark that many other results can be obtained once Theorem 1 (or rather the spectral picture, based on Lemma \ref{lem:lasota-yorke}, from which the Theorem follows) is established. For example, approximation results (extending the ones in \cite{Mu}) follow by applying the theorems in \cite{KL2,KL3}. Also various limit theorems (in particular the Central Limit Theorem as in \cite{Li4}) can be obtained by using, e.g., the results in \cite{Gou1}. Finally, the interested reader can easily extend the present arguments to more general transfer operators by imposing appropriate conditions on the potentials, in the spirit of \cite{Ry}.
\end{rem}
\subsection{An example}
Here is an example (chosen more or less at random). Let $\Omega\subset (0,1)^2$ be given by the smoothness domains of the map
\[
T(x,y)=(a\sqrt x,b y+\sqrt x)-\lfloor(a\sqrt x,b y+\sqrt x)\rfloor,
\]
where $a,b>6$ and $\lfloor c\rfloor$ is the integer part of $c$.
Then
\[
(DT)^{-1}=\begin{pmatrix}\frac {2\sqrt x}a & 0\\
-\frac 1{ab}& \frac 1b\end{pmatrix}.
\]
It is easy to check that the above map satisfies all the conditions of Theorem \ref{thm:main}.

\begin{proof}[{\bf Proof of Theorem \ref{thm:main}}]
As usual, one first defines the transfer operator and then proves a Lasota-Yorke inequality. More precisely, for each $x\in \Omega$, 
\begin{equation}\label{eq:transfer}
\cL h(x):=\sum_{y\in T^{-1} (x)}|\det D_yT|^{-1} h(y).
\end{equation}
Then for each $h\in L^1(\bR^d,m)$ and $\vf\in L^\infty(\bR^d,m)$ we have\footnote{ In fact \eqref{eq:transfer1} is often taken as the definition of $\cL$. Formula \eqref{eq:transfer} follows then by changing variables, remembering hypothesis {\bf\ref{c:0}} and using the Lebesgue monotone convergence Theorem to prove almost sure convergence of the right hand side.}
\begin{equation}\label{eq:transfer1}
\int_\Omega h\cdot \vf\circ T=\int_\Omega \vf\cdot \cL h.
\end{equation}
Hence, $\Id_\Omega\cL h=\cL(\Id_\Omega h)$ a.s..
The Lasota-Yorke inequality is proven in Lemma \ref{lem:lasota-yorke}. It follows that $\cL$ is a bounded operator from $L^1(\Omega,m)$, and from $BV(\Omega)$, to itself.

Using standard arguments (see \cite{Babo} for a general discussion, here I use Hennion's Theorem in \cite{He}) Lemma \ref{lem:lasota-yorke} implies that $\cL$ is a quasi-compact operator with spectral radius one and essential spectral radius bounded by $\lambda^{-1}$. In addition, one is an eigenvalue since $\int h=\int \cL h$, for each $h\in L^1$, implies that one belongs to the spectrum of the dual operator $\cL'$. Thus there exists a spectral gap between the eigenvalues of modulus one and the rest of the spectrum. Accordingly, we can write $\cL=\Pi+R$ where $\Pi$ is a finite rank operator with spectrum contained in $\{z\in\bC\;:\; |z|=1\}\cup\{0\}$ while the spectral radius of $R$ is strictly smaller than one and $\Pi R=R\Pi=0$. In addition, the second of the \eqref{eq:lasota} shows that $\Pi$ must be power bounded, hence it cannot contain Jordan blocks. Thus, we can further decompose $\Pi=\sum_je^{i\theta_j}\Pi_{\theta_j}$, $\Pi_{\theta_j}\Pi_{\theta_k}=\delta_{jk} \Pi_{\theta_j}$.\footnote{ Note that one should complexify BV. I do not give details as they are standard.} By the spectral gap
\begin{equation}\label{eq:decomp}
\lim_{n\to\infty}\frac 1n \sum_{k=0}^{n-1}e^{-i\theta k}\cL^k=\begin{cases} \Pi_{\theta_j}\quad&\text{if } \theta=\theta_j\\
0&\text{otherwise,}
\end{cases}
\end{equation}
where the limit is in the strong operator topology in $L(BV,BV)$. Again, the second of \eqref{eq:lasota} implies that the $\Pi_{\theta_j}$ are bounded operators from $L^1$ to $BV$. In particular, $\Pi_0 1\in BV$ is the density of an invariant measure. On the other hand, if $\mu$ is an absolutely continuous invariant measure with density $h$, \eqref{eq:decomp} implies that $h=\Pi_0 h$, hence $h\in BV$. That is, all the invariant densities belong to the finite dimensional range of $\Pi_0$.

Next, observe that $\Pi_{\theta_j}(g)=\sum_\kappa\ell_{j,\kappa}(g) h_{j,\kappa}$, where $h_{j,\kappa}\in BV$ and $\ell_{j,\kappa}\in (L^1)'$ and $\ell_{j,\kappa}(h_{r,s})=\delta_{j,r}\delta_{\kappa,s}$. That is, there exists $\rho_{j,\kappa}\in L^\infty$ such that $\ell_{j,\kappa}(g)=\int_\Omega \rho_{j,\kappa}\cdot g$, for all $g\in L^1$. Moreover, $\Pi_{\theta_0}:=\Pi_0$ is a positive projector and $\int \Pi_0 h=\int h$, hence $\sum_\kappa m(h_{0,\kappa})\rho_{0,\kappa}=1$, which implies $\rho_{0,\kappa}=\Id_{A_\kappa}m(h_{0,\kappa})^{-1}$, where the $\{A_\kappa\}$ are disjoint indecomposable invariant sets\footnote{ Remember that, by construction, $\cL\Pi_{\theta_j}= \Pi_{\theta_j}\cL=e^{i\theta_j}\Pi_{\theta_j}$ which implies $\cL'\ell_{0,\kappa}=\ell_{0,\kappa}$ and thus $\rho_{0,\kappa}\circ T=\rho_{0,\kappa}$ a.s.. In turns, this means that $\rho_{0,\kappa}$ can take only one value (besides zero) otherwise we could construct new invariant sets, and hence invariant measures, which we know do not exists. For the same reasons we must have $\rho_{0,\kappa}\rho_{0,j}=\delta_{\kappa, j}\rho_{0,\kappa}^2$.} and must therefore correspond to the ergodic decomposition of Lebesgue. In addition, for all $g\in L^1$,
\[
\begin{split}
&\int \rho_{j,\kappa}\rho_{j',\kappa'} \cL g=\int (\rho_{j,\kappa}\rho_{j',\kappa'})\circ T \cdot g=e^{i(\theta_j+\theta_j')} \int \rho_{j,\kappa}\rho_{j',\kappa'}g\\
&\int \bar \rho_{j,\kappa} \cL g=e^{-i\theta_j}\int \bar \rho_{j,\kappa} g.
\end{split}
\]
Hence $e^{-i\theta_j}\in\sigma(\cL')=\sigma(\cL)$ and either $\rho_{j,\kappa}\rho_{j',\kappa'}=0$ or $e^{i(\theta_j+\theta_{j'})}\in\sigma(\cL)$. I.e., the peripheral eigenvalues form finitely many cyclic groups, each related to a different ergodic component (see \cite{Babo} for details).

If the peripheral spectrum consist of one as a simple eigenvalue then $\Pi$ is a rank one operator and setting $h_*=\Pi 1$ for each $h\in BV$ and $f\in L^d$ holds (remembering \eqref{eq:sobo})
\[
\int_\Omega f\circ T^n h=\int _\Omega f\cL^n h=\int_\Omega f h_* \cdot \int_\Omega h +\cO(\|f\|_{L^d}\| R\|_{BV}^n \|h\|_{BV})
\]
which gives the announced exponential decay of correlations. Similar arguments work for $(T^N, \widetilde \Omega_j, \mu)$ in the general case.

The only thing left to prove is the openness of the ergodic components. If $d=1$, then it follows from the fact that, in one dimension only, $BV$ functions are strictly positive over some ball (i.e. interval) and by the standard argument in the last paragraph of the proof of Lemma \ref{lem:open}. If $d>1$, then the result follows from hypothesis {\bf \ref{c:6}} and is proven in Lemma \ref{lem:open}.
\end{proof}
\begin{rem} In the following we will use $C_\#$ to designate a generic constant, depending only on the map $T$ and the constants appearing in the hypotheses  {\bf \ref{c:0}-\ref{c:6}}, which values can change from line to line. On the contrary we will use $C_{a,b,c,\dots}$ for a generic constant depending on the parameters $a,b,c,\dots$.
\end{rem}

\section{Lasota-Yorke inequality}\label{sec:lasota}
\begin{lem}\label{lem:lasota-yorke}
For each map $T$ satisfying hypotheses {\bf \ref{c:0}-\ref{c:3}} and for each $\sigma\in (\lambda^{-1},1)$ there exists $B>0$ such that, for all $n\in\bN$ and $h\in BV$, holds true
\begin{equation}\label{eq:lasota}
\begin{split}
&\|\cL^nh\|_{L^1(\Omega,m)}\leq \|h\|_{L^1(\Omega,m)}\,,\\
&\|\cL^nh\|_{BV(\Omega)}\leq \sigma^{n} \|h\|_{BV(\Omega)}+B\|h\|_{L^1(\Omega,m)}.
\end{split}
\end{equation}
\end{lem}
\begin{proof}
For each $h\in L^1(\Omega,m)$ and $\vf\in L^\infty(\bR^d, m)$, we have
\[
\int_{\Omega}\cL h\cdot \vf=\int_{\Omega}h\cdot \vf\circ T\leq \|h\|_{L^1}\|\vf\|_{L^\infty}.
\]
Thus $\|\cL h\|_{L^1}\leq \| h\|_{L^1}$.\footnote{ Note that equality holds if $h\geq 0$.} 

To prove the second of \eqref{eq:lasota} let $h\in BV$ and $\vf\in\cC^1_0(\bR^d,\bR^d)$ with $\|\vf\|_{L^\infty}\leq 1$,
\begin{equation}\label{eq:step1}
\begin{split}
\int_{\Omega}&\cL h\cdot\sum_{i=1}^d \partial_{x_i}\vf_i=\int_{\Omega}h\; \sum_{i=1}^d(\partial_{x_i}\vf_i)\circ T=
\sum_{i,k}\int_{\Omega_k}h\; (\partial_{x_i}\vf_i)\circ T\\
&=\sum_{i,k,j}\int_{\Omega_k}h\;\partial_{x_j}\{[(DT)^{-1}]_{ji}[ \vf_i \circ T]\}-\sum_{i,k,j}\int_{\Omega_k}h\;[ \vf_i \circ T]\partial_{x_j}[(DT)^{-1}]_{ji}\,.
\end{split}
\end{equation}
To proceed we need to recreate proper test functions. To this end remember from hypothesis {\bf \ref{c:4}} that for each $j\in\{1,\dots, d\}$ and $x_{\neq j}\in \bR^{d-1}$ we have defined intervals  $I_{k,\ell}(x_{\neq j})=:(a_{k,\ell}(x_{\neq j}), b_{k,\ell}(x_{\neq j}))$. Let us define $J_k(x_{\neq j}):=\cup_\ell\overline{ I_{k,\ell}(x_{\neq j})}$ and, for each $x\in\bR^d$,\footnote{ The numbers $\Psi_{j,k,\ell}^\pm(x_{\neq j})$ are well defined by hypothesis {\bf \ref{c:2}}. In addition, if $x\in \{x_{\neq j}\}\times\overline{J_k(x_{\neq j})}$ but not in $\{x_{\neq j}\}\times J_k(x_{\neq j})$, then $x$ is a point of accumulation of intervals and hence, by hypothesis {\bf \ref{c:4}}, is natural to define $\Psi_{j,k}(x)=0$.}
\[
\begin{split}
&\Psi_{j,k}(x)=\sum_{i=1}^d[(D_xT)^{-1}]_{ji}\cdot \vf_i\circ T(x)\cdot \Id_{J_k(x_{\neq j})}(x_j)\\
&\Psi_{j,k,\ell}^-(x_{\neq j})=\Psi_{j,k}(x_1,\dots, x_{j-1},a_{k,\ell}, x_{j+1},\dots, x_d)\\
&\Psi_{j,k,\ell}^+(x_{\neq j})=\Psi_{j,k}(x_1,\dots, x_{j-1},b_{k,\ell}, x_{j+1},\dots, x_d).
\end{split}
\]
Next, define functions $\eta_{j,k,\ell, x_{\neq j}}\in L^\infty(\bR,m)$ by
\[
\eta_{j,k,\ell, x_{\neq j}}(y)=\begin{cases} 0 &\forall\; y\in (-\infty, a_{k,\ell}-\delta]\\
                                        \Psi_{j,k,\ell}^-(x_{\neq j})\cdot (y-a_{k,\ell}+\delta)\delta^{-1}  &\forall \;y \in(a_{k,\ell}-\delta, a_{k,\ell})\\
                                        0 &\forall\; y\in [a_{k,\ell},b_{k,\ell}]\\
                                        \Psi_{j,k,\ell}^+(x_{\neq j})\cdot(b_{k,\ell}+\delta-y)\delta^{-1}  &\forall \; y \in(b_{k,\ell}, b_{k,\ell}+\delta)\\
                                        0 &\forall\; y\in  [b_{k,\ell}+\delta, +\infty),
                            \end{cases}
\]
where $\delta>0$ satisfies hypothesis {\bf \ref{c:4}}, and
\[
\begin{split}
&\overline\Psi_{j,k}(x)=\Psi_{j,k}(x)+\sum_{\ell}\eta_{j,k,\ell, x_{\neq j}}(x_j)=\sum_{\ell}\theta_{k,\ell}\;,\\
&\theta_{k,\ell}=\left[\Psi_{j,k}(x)\Id_{\overline{I_{k,\ell}(x_{\neq j})}}(x_j)+\eta_{j,k,\ell, x_{\neq j}}(x_j)\right]\;;\quad\Theta_{j}:=\sum_{k}\overline \Psi_{j,k}.
\end{split}
\]
The point of introducing such functions is that, by construction, the $\theta_{k,\ell}$ are continuous functions in the $x_j$ variable, for each $x_{\neq j}$, and the same holds for the functions $\Theta_{j}$ which are, by definition and hypothesis {\bf \ref{c:4}}, a uniformly convergent series of continuous functions.
Also, for $x\not\in\partial\Omega$, {\bf \ref{c:4}} implies 
\begin{equation}\label{eq:thetaj}
\sup_j|\Theta_j(x)|\leq\sup_j\sum_{k,\ell}\left[\sup_{x_j\in I_{k,\ell}(x_{\neq j})}|\Psi_{j,k}(x)|\right]\Id_{[a_{k,\ell}-\delta,b_{k,\ell}+\delta]}(x_j)\leq \lambda^{-1}.
\end{equation}
Moreover,\footnote{ The  first equality holds by hypotheses {\bf \ref{c:1}}, {\bf \ref{c:2}}, the second follows from hypotheses {\bf \ref{c:4}}, {\bf \ref{c:5}}, {\bf \ref{c:3}} and the Lebesgue dominated convergence Theorem.}
\[
\Theta_j(x)=\sum_{k,\ell}\int_0^{x_j}\hskip-.4cm dy\;\;\partial_y\theta_{k,\ell}(x_{\neq j},y)=\int_0^{x_j}\hskip-.4cm dy\;\;\sum_{k,\ell}\partial_y\theta_{k,\ell}(x_{\neq j},y)
\]
implies that the function $\Theta_j$ is Lebesgue almost surely differentiable with respect to the variable $x_j$ and  $\partial_{x_j}\Theta_j(x_{\neq j},\cdot) \in L^1(\bR)$. Thus, almost surely,\footnote{ Note that the last term of the next equation is bounded by $d\delta^{-1}\lambda^{-1}$ due to hypothesis {\bf \ref{c:4}}.}
\[
\begin{split}
\div\Theta(x)=&\sum_{i,k,j}\partial_{x_j}\{[(D_xT)^{-1}]_{ji}\vf_i\circ T\}\cdot  \Id_{J_k(x_{\neq j})}(x_j)\\
&+\delta^{-1}\sum_{j,k,\ell}\left\{\Psi_{j,k,\ell}^-(x_{\neq j})\Id_{[a_{k,\ell}-\delta,a_{k,\ell}]}(x_j)-\Psi_{j,k,\ell}^+(x_{\neq j})\Id_{[b_{k,\ell},b_{k,\ell}+\delta]}(x_j)\right\}.
\end{split}
\]
This implies, by \eqref{eq:legal}, that $\Theta$ has enough properties to be used as a test function.

Having set up the above machinery we rewrite \eqref{eq:step1} as
\begin{equation}\label{eq:step2}
\begin{split}
\left|\int_\Omega \!\!\cL h \cdot \div \vf\right|&\leq\left|\int_{\bR^d} \!\! \!\! h\,\div \Theta\right|+\frac {d\|h\|_{L^1}}{\lambda\delta}+\left|\sum_{i,k,j}\int_{\Omega_k} \!\! \!\!h\;[ \vf_i \circ T]\partial_{x_j}[(DT)^{-1}]_{ji}\right|\\
&\leq \frac{\|h\|_{BV}}{\lambda}+\frac {d\|h\|_{L^1}}{\lambda\delta} +\left|\sum_{i,k,j}\int_{\bR^d}\hskip-6pt h\;[ \vf_i \circ T]\partial_{x_j}[(DT)^{-1}]_{ji}\Id_{\Omega_k}\right|.
\end{split}
\end{equation}
To conclude note that one can split the last term of \eqref{eq:step2} as
\[
\sum_{i,k,j}\int_{\bR^d}\hskip-.2cm h\;[ \vf_i \circ T]\partial_{x_j}[(DT)^{-1}]_{ji}\Id_{\Omega^\ve_k}+
\sum_j\int_{\bR^d}\hskip-.2cm h\;\partial_{x_j}\hskip-.2cm\int_0^{x_j}\sum_{i,k}[ \vf_i \circ T]\partial_{x_j}[(DT)^{-1}]_{ji}\Id_{\partial^\ve\Omega_k}.
\]
Thanks to hypothesis {\bf \ref{c:5}}, for each $\sigma\in (\lambda^{-1},1)$, one can then chose $\ve$ so that 
\[
\left\|\bigg(\int_0^{x_j}dy\sum_{i,k,j}\vf_i \circ T(x_{\neq j}, y)\cdot \partial_{y}[(D_{(x_{\neq j}, y)}T)^{-1}]_{ji}\cdot \Id_{\partial^\ve\Omega_k}(x_{\neq j}, y)\bigg)\right\|_{L^\infty}\leq\sigma-\lambda^{-1}.
\]
Finally, note that if $\Omega_k^\ve\neq\emptyset$, then $\Omega_k$ must contain a ball of radius $\ve$, but only finitely many such balls can fit in $\Omega$, thus only finitely many $\Omega_k^\ve$ are non empty. Accordingly, by hypothesis {\bf\ref{c:1}} and \eqref{eq:legal}, there exists $C_\sigma>0$ such that, taking the $\sup$ over $\vf$,
\[
\| \cL h \|_{BV}\leq\lambda^{-1}\| h\|_{BV}+d\delta^{-1}\lambda^{-1}\|h\|_{L^1}+(\sigma-\lambda^{-1})\| h\|_{BV}+C_\sigma\|h\|_{L^1},
\]
which (iterating) proves the Lemma with $B=(1-\sigma)^{-1}(C_\sigma+d\delta^{-1}\lambda^{-1})$.
\end{proof}

\section{Openness of the ergodic components}\label{sec:open}
Finally, we deal with the openness of the ergodic components. As far as I know the only general result of this type (but limited to finite partitions and bounded derivatives) is in \cite{Sa} where it can be easily obtained thanks to the particular Banach space used there. Yet, I do not see how to use such a Banach space in the unbounded derivative case. In addition, the argument in \cite{Sa} does not seem to be constructive. On the contrary the radius $\ve_*$ in the following theorem is constructive and could, with some extra work, be estimated explicitly in concrete examples.

\begin{lem}\label{lem:open} For each map $T$ satisfying hypotheses {\bf \ref{c:0}-\ref{c:6}} there exists $\ve_*>0$ such that each ergodic component of the Lebesgue measure is open and contains a ball of radius $\ve_*$ on which the density of the associated absolutely continuous invariant measure (a.c.i.m.) is uniformly positive.
\end{lem}
\begin{proof}
We argue along known lines (see \cite{Ke1} for similar arguments in a closely related case). The basic idea is to compare the true dynamics with one in which only trajectories that are always far away from the singularities are considered. Such trajectories experience only the smooth part of the dynamics. To do so we need first to define a transfer operator restricted to such trajectories and then to obtain an a priori estimate on the difference between the real system and the {\em restricted} one. 

Consider the sets  $\Gamma_{\ve}=\cup_k\partial^\ve\Omega_k$, and the operators $\cL_{\ve,n}f=\cL((1-\Id_{\Gamma_{\ve\nu^n}})f)$,\footnote{ This corresponds to an open dynamics where the trajectory disappears when it gets closer than $\nu^n\ve$ to the boundary of an $\Omega_k$.} where $\nu\in(\nu_0^{\frac 1a},1)$, $\nu_0=\|DT^{-1}\|_{L^\infty}$, (the interval is not empty by assumption {\bf \ref{c:4}}). Setting $\widetilde\cL_{\ve,n}=\cL_{\ve,0}\cL_{\ve,1}\cdots\cL_{\ve,n-1}$, $\widetilde\cL_{\ve,0}=\Id$, for each $f\in BV$, $f\geq 0$, we have\footnote{ In the second line we have equality since $\cL$ is an $L^1$ isometry on positive functions, in the third we have used the H\"older inequality \eqref{eq:sobo} and hypothesis {\bf\ref{c:6}}. The forth line follows from Lemma \ref{lem:lasota-yorke}.}
\begin{equation}\label{eq:boundary1}
\begin{split}
\|\cL^n f-&\widetilde\cL_{\ve,n}f\|_{L^1}\leq \sum_{k=0}^{n-1}\| \cL^{k+1}\cL_{\ve,k+1}\cdots\cL_{\ve,n-1} f-\cL^{k}\cL_{\ve,k}\cdots\cL_{\ve,n-1} f\|_{L^1}\\
&= \sum_{k=0}^{n-1}\|\Id_{\Gamma_{\ve\nu^k}}\cdot\cL_{\ve,k+1}\cdots\cL_{\ve,n-1} f\|_{L^1}
\leq \sum_{k=0}^{n-1}\|\Id_{\Gamma_{\ve\nu^k}}\cdot\cL^{n-k-1} f\|_{L^1}\\
&\leq\sum_{k=0}^{n-1}m(\Gamma_{\ve\nu^k})^{\frac 1d}\|\cL^{n-k-1}f\|_{L^{\frac d{d-1}}}
\leq\sum_{k=0}^{n-1}C_\#\ve^{\frac {\alpha } d}\nu^{\frac {\alpha k}d}\|\cL^{n-k-1}f\|_{BV}\\
&\leq\frac {C_\# \ve^{\frac \alpha d}}{1-\nu^{\frac \alpha d}}\|f\|_{BV}.
\end{split}
\end{equation}

Next, we choose some mollifier $\eta$ and, for each $\delta\in(0,1)$, we introduce the regularization $h_\delta$ of the invariant measure as defined in Appendix \ref{sec:regular}, we will be interested in estimates uniform in $\ve$ for $\delta$ arbitrarily small.

Suppose that, for some $x_0\in \Omega\setminus \Gamma_\ve$, $\widetilde\cL_{\ve,n} 1(x_0)>0$. This means that, for each $y\in T^{-n}\{x_0\}$ that contributes a non zero term to the sum defining $\widetilde\cL_{\ve,n}$, it holds true $T^k y\not\in \Gamma_{\ve\nu^{n-k-1}}$ for each $k\in\{0,\dots,n-1\}$. Hence, if $\|x_0-z\|\leq \ve/2$, it follows that there exists $w\in T^{-n}\{z\}$ such that, for each $k\in \{0,\dots, n-1\}$, $\|T^k w-T^k y\|\leq \nu_0^{n-k-1}\ve/2$, hence $T^k w\not\in \Gamma_{\ve\nu^{n-k-1}/2}$. 
A standard distortion argument together with \eqref{eq:disto}) yields
\[
\begin{split}
|\det(D_wT^n)^{-1}h_\delta(w)|&= e^{-\sum_{k=0}^{n-1}\ln|\det(D_{T^kw}T)|}\;h_\delta(w)\\
&\geq e^{-C_\eta\{\sum_{k=0}^{n-1}\nu_0^{(n-k)}\nu^{-(n-k)a}\ve^{1-a}+\nu_0^n\ve \delta^{-d-1}\}}|\det(D_yT^n)^{-1}| h_\delta(y),
\end{split}
\]
where we have used assumption {\bf \ref{c:6}}. The above implies that, for each $\ve>0$, there exists $\gamma_\ve>0$ such that, for each $n\in\bN$ and $\delta \geq \delta_n=\nu_0^{\frac{n}{d+1}}$, holds true
\begin{equation}\label{eq:ed-low}
\cL^{n} h_\delta(z)\geq \widetilde\cL_{\ve/2,n} h_\delta(z)\geq \gamma_\ve\widetilde\cL_{\ve,n} h_\delta(x_0).
\end{equation}

It remains to find good points $x_0$. Let $A=\{x\in\Omega\;:\; h(x)\geq \frac 12\}$, then we have $m(A)>C_\#$.\footnote{ Indeed,
\[
1=m(\Id_A\, h)+m(\Id_{A^c}\, h)\leq C_\#\|h\|_{BV}m(A)^{\frac 1d}+\frac 12m(A^c).
\]
Thus, $1\leq 2C_\#\|h\|_{BV}m(A)^{\frac 1d}-m(A)$, which implies $m(A)\geq (2C_\#\|h\|_{BV})^{-d}$. But the second of \eqref{eq:lasota} implies $\|h\|_{BV}=\|\cL h\|_{BV}\leq \sigma\|h\|_{BV}+B$. Hence, $m(A)\geq (2C_\# B)^{-d}(1-\sigma)^d$.}
Consider $B_{\delta,n}=\{x\in A\;:\; \widetilde\cL_{\ve,n} h_\delta(x)\leq \frac 14\}$, then by \eqref{eq:boundary1} and Lemma \ref{lem:L1delta}
\[
\frac 14 m(B_{\delta,n})\geq \|\Id_{B_{\delta, n}} h\|_{L^1}-\|h- \widetilde\cL_{\ve,n} h_\delta\|_{L^1}\geq \frac 12 m(B_{\delta,n})-C_\eta(\delta^{\frac\alpha d}+\ve^{\frac\alpha d}).
\]
Thus, provided $\delta, \ve$ are both small enough, we have $m(A\setminus (B_{\delta,n}\cup\Gamma_\ve))\geq \frac 12m(A)$ for all $n\in\bN$. 

To conclude, set $\tilde B_{n,\ve}=B_{\delta_n,n}\cup\Gamma_\ve$. The above implies that there exist a positive measure set $D\subset A$ such that for each $x_0\in D$ there exists an infinite sequence $\{n_j\}$ such that $x_0\in A\setminus \tilde B_{n_j,\ve}$ for all $j\in\bN$.\footnote{ By the monotone convergence Theorem 
$\int_\Omega \sum_{n=0}^\infty \Id_{A\setminus \tilde B_{n,\ve}}=\sum_{n=0}^\infty m(A\setminus \tilde B_{n,\ve})=\infty$,
thus the sum must diverge on a positive measure set.}
It follows that for all $x_0\in D$ and all $z$ in an $\ve/2$ neighborhood of $x_0$, by equation \eqref{eq:ed-low}, we have $\cL^{n_j} h_{\delta_{n_j}}(z)\geq \frac 14 \gamma_\ve$.\footnote{ Since $x_0\in A\setminus \tilde B_{n_j,\ve}$, it must be $\widetilde \cL_{\ve,n_j}1(x_0)>0$.}
On the other hand $f_j:=\cL^{n_j} h_{\delta_{n_j}}$ converges to $h$ in $L^1$,\footnote{ In fact, $\|h-f_j\|_{L^1}=\|\cL^{n_j}(h-h_{\delta_{n_j}})\|_{L^1}\leq\|h-h_{\delta_{n_j}}\|_{L^1} \leq C_\eta\delta^{\frac\alpha d}_{n_j}\|h\|_{BV}$, by  Lemma \ref{lem:L1delta}.} hence there exists a subsequence $f_{j_k}$ which converges to $h$ almost surely. Finally, this implies $h\geq  \frac 14 \gamma_\ve$ almost surely in the ball $\{z\in \Omega\;:\; \|z-x_0\|\leq \ve/2\}$, which shows that the support of any invariant measure must contain an open ball.

We are left with the task of proving the openness of the ergodic components. Given an ergodic component $\Delta\subset \Omega$ let $h$ be the density of the a.c.i.m. $\mu$ supported on it. By the previous arguments we know that there exists an open ball $B\subset \Delta$, such that $\Id_B\leq C_\#h$. Then $\cL^n\Id_B\leq C_\#\cL^n h=C_\# h$ implies that $\Lambda_0:=\cup_{n=0}^\infty T^nB$ is contained in the support of $h$. Next, notice that for all connected open sets $U\subset \Omega$, $U_1=U\setminus T^{-1}\partial\Omega$ is open.\footnote{ Indeed, since $U$ is connected it must be contained in some $\Omega_k$, then $T_k=T|_{\Omega_k}$ is continuous by hypothesis hence  $ T^{-1}(\partial\Omega)\cap\Omega_k=  T_k^{-1}(\partial\Omega)$ is closed.} Accordingly, $TU_1\subset \Omega$ is open which, by hypothesis {\bf\ref{c:0}}, shows that $TU$ differs from an open set by a zero measure set. Iterating this argument\footnote{ By considering a connected component at a time.} shows that there exists an open set $V\subset \Lambda_0$ such that $\mu(\Lambda_0\setminus V)=0$. On the other hand, $T^{-1}\Lambda_0\supset \Lambda_0$, hence $\mu(T^{-1}\Lambda_0\setminus\Lambda_0)=0$. This implies that there exists $\Lambda_1\supset \Lambda_0$, $\mu(\Lambda_1\setminus\Lambda_0)=0$, such that $T^{-1}\Lambda_1=\Lambda_1$, hence $0=\mu(\Delta\setminus\Lambda_1)=m(\Delta\setminus V)$ which implies the Lemma.
\end{proof}

\appendix
\section{BV functions}\label{sec:bv}

Here we collect some basic facts about BV functions.\footnote{ For the general theory of BV functions see \cite{EG}, or see \cite{KL1} for a quick introduction to the properties relevant to the present context.}

\begin{lem}\label{lem:bv} There exists $C_d>0$ such that, for each $h\in BV(\Omega)$, $A\subset \bR^d$,
\begin{equation}\label{eq:sobo}
\|\Id_A h\|_{L^1}\leq m(A)^{\frac 1d}\|h\|_{L^{\frac{d}{d-1}}}\leq C_d m(A)^{\frac 1d} \|h\|_{BV}.
\end{equation}

The set $B_1:=\{h\in BV\;:\; \| h\|_{BV}\leq 1\}$ is relatively compact in the $L^1$ topology.

For all $j\in\{1,\dots, d\}$ and almost all $x_{\neq j}$, $h(x_{\neq j},\cdot)\in BV(\bR)\subset L^{\infty}(\bR, m)$.

For each $h\in BV(\bR^d)$ and all $\vf\in L^\infty(\bR^d, m)$ of compact support and  such that, $x_{\neq j}$ a.s., $\vf_j(x_{\neq j}, \cdot)\in\cC^0(\bR,\bR)$, $(\partial_{x_j}\vf_j)(x_{\neq j}, \cdot)\in L^1(\bR, m)$ holds true
\begin{equation}\label{eq:legal}
\left|\int_{\bR^d} h\,\div \vf\right|\leq \|h\|_{BV(\bR^d)}\|\vf\|_{L^\infty}.
\end{equation}
\end{lem}
\begin{proof}
The first inequality of \eqref{eq:sobo} is H\"older inequality:
\[
\int_A |h|\leq m(A)^{\frac1d}\|h\|_{L^{\frac{d}{d-1}}},
\] 
while the second is a Sobolev type inequality, see \cite[section 5.6.1, Theorem 1, (i)]{EG}.

The second statement follows directly from \cite[section 5.2.3, Theorem 4]{EG} since $B_1\subset\{h\in BV((-1,2)^d)\;:\; \|h\|_{BV((-1,2)^d)}\leq 1\}$.

To prove the third statement let 
\[
V_jh(x_{\neq j})=\sup_{\substack{\vf\in\cC^1_0(\bR,\bR)\\|\vf|_{L^\infty}\leq 1}}\int_{\bR}h(x_{\neq j}, t)\vf'(t) dt.
\]
Then \cite[section 5.10.2, Theorem 2]{EG} states that $V_jh\in L^1(\bR^{d-1},m)$. Hence, $V_j h$ must be almost surely finite, from which the statement follows (taking into account \eqref{eq:sobo} for $d=1$). 

To prove \eqref{eq:legal} we have to deal with the fact that $\vf$ is not a legal test function since it may be discontinuous. Nevertheless, this is a superficial problem: given a positive function $\eta\in\cC^\infty(\bR,\bR)$ with integral one and supported in $[-1,1]$ define $\eta_\ve(y)=\ve^{-1}\eta(\ve^{-1}y)$, $y\in\bR$, and $\bar\eta_{j,\ve}(x)=\prod_{\substack{k=1\\ k\neq j}}^d\eta_\ve( x_k)$ and  let $*$ stand for the usual convolution.
We then define $\vf_{\ve_1,\ve_2}\in\cC^\infty(\bR^d,\bR^d)$ by\footnote{ The second convolution is with respect the variable $x_j$.} $(\vf_{\ve_1,\ve_2})_j:=\bar\eta_{j,\ve_1}*(\eta_{\ve_2}*\vf_{j})$. Note that $\|\vf_{\ve_1,\ve_2}\|_{L^\infty}\leq \|\vf\|_{L^\infty}$. For each $h\in BV(\bR^d)$, the second statement of the Lemma implies, $x_{\neq j}$ almost surely,\footnote{The integration by part in the last line is allowed since $\vf_j$ is, $x_{\neq j}$ almost surely, of bounded variation in the $x_j$ variable.}  
\[
\begin{split}
 \int \!\!dx_j h(x_{\neq j}, x_j) \partial_{x_j}\vf_{j}(x_{\neq j}, x_j)&=\lim_{\ve_2\to 0}\int \!\! dx_j h(x_{\neq j}, x_j) \int_{\bR}dy\partial_y\vf_j(x_{\neq j},y)\eta_{\ve_{2}}(x_j-y)\\
 &=\lim_{\ve_2\to 0} \!\int  \!\! dx_j h(x_{\neq j}, x_j) \partial_{x_j} \!\!\int_{\bR}dy\vf_j(x_{\neq j},y)\eta_{\ve_{2}}(x_j-y).
\end{split}
\]
In addition, for each $\ve_2>0$, the right hand side is bounded by $V_j h \|\vf_j\|_{L^\infty}$ which, by
the above results, belongs to $L^1(\bR^{d-1},m)$. Hence, by Lebesgue dominate convergence Theorem,
\[
\begin{split}
\int_{\bR^d} h\; \partial_{x_j}\vf_j&=\lim_{\ve_2\to 0}\int_{\bR^d} h\; \partial_{x_j}(\eta_{\ve_2}*\vf_j)=\lim_{\ve_2\to 0}\lim_{\ve_1\to 0}\int_{\bR^d} \bar \eta_{j,\ve_1}*h\; \cdot( \partial_{x_j}\eta_{\ve_2}*\vf_j)\\
&=\lim_{\ve_2\to 0}\lim_{\ve_1\to 0}\int_{\bR^d} h\; \cdot\partial_{x_j}[\bar \eta_{j,\ve_1}*( \eta_{\ve_2}*\vf_j)]\,.
\end{split}
\]
Thus,
\[
\begin{split}
\left|\int_{\bR^d} h\; \div\vf\right|&= \lim_{\ve_2\to 0} \lim_{\ve_1\to 0}\left|\int_{\bR^d} h\; \div\vf_{\ve_1,\ve_2}\right |\leq\lim_{\ve_2,\ve_1\to 0}\|h\|_{BV(\bR^d)}\|\vf_{\ve_1,\ve_2}\|_{L^\infty}\\
&\leq \|h\|_{BV(\bR^d)}\|\vf\|_{L^\infty}.
\end{split}
\]
\end{proof}

\section{A regularization scheme}\label{sec:regular}
Here I describe a way to regularize a $BV(\Omega)$ density and some relevant related properties. This is fairly standard, I add it here just for the reader convenience.

Let us fix $\eta\in\cC^\infty(\bR,\bR_+)$, $\operatorname{supp} \eta\subset [-1,1]$, $\int\eta=1$ and set, for $\delta>0$, $\bar \eta_\delta(x)=\delta^{-d}\prod_{i=1}^d\eta(\delta^{-1}x_i)$. Let $h\in BV(\Omega)$ be such that, $h\geq 0$, $\int h=1$. For each $\delta>0$ define $h_\delta(x)=\int_{\Omega}\bar \eta_\delta(x-y)h(y)dy +\delta$. 

\begin{lem}\label{lem:L1delta} If $\Omega$ satisfies the first inequality of hypothesis {\bf\ref{c:6}}, then the following properties hold true:
\[
\begin{split}
&h_\delta\geq \delta\;,\\
& \|\nabla h_\delta\|_{L^1}\leq \|h\|_{BV}\;, \\
&\|\nabla h_\delta\|_{L^\infty}\leq C_\eta\delta^{-d}\|h\|_{BV}\;,\\
&\|h-h_\delta\|_{L^1}\leq C_\eta \delta^{\frac \alpha d}\|h\|_{BV} .
\end{split}
\]
\end{lem}
\begin{proof}
The first inequality is trivial.
The second follows from
\[
\begin{split}
\|\nabla h_\delta\|_{L^1}&=\sup_{\substack{\vf\in\cC_0^1\\ \|\vf\|_{L^\infty}\leq 1}}\sum_{k=1}^d\int_{\bR^d}\hskip-.2cm dx\; \vf_k(x) \int_{\Omega}\hskip-.2cm dy\;\partial_{x_k} \bar \eta_\delta(x-y) h(y)\\
&=\sup_{\substack{\vf\in\cC_0^1\\ \|\vf\|_{L^\infty}\leq 1}}\sum_{k=1}^d\int_\Omega h(x) \partial_{x_k}(\bar \eta_\delta*\vf_k)(x).
\end{split}
\]
The next inequality follows from $\|\bar \eta_\delta\|_{L^\infty}\leq C_\eta\delta^{-d}$. To prove the last inequality, let $\chi_\delta\in\cC^1(\bR^d,[0,1])$ such that $\operatorname{supp}\chi_\delta\subset \Omega^{\delta/ 2}$ and $\operatorname{supp}(1-\chi_\delta)\subset  (\Omega^\delta)^c$. Then, by \eqref{eq:sobo},
$\|h-h_\delta\|_{L^1}\leq \|\chi_\delta(h-\bar \eta_\delta*h)\|_{L^1}+\delta\|h\|_{L^1}+C\delta^{\frac \alpha d}\|h\|_{BV}$. Next, for each $g\in \cC^1(\bR^d,\bR)$, and $\vf\in\cC^0(\Omega,\bR)$, let $\vf_\delta=\chi_\delta\vf$ and
\[
\begin{split}
\int_{\Omega}dx&\int_{ \bR^d}dy\vf_\delta(x) \bar \eta_\delta(x-y)(g(y)-g(x))\\
&=\int_{\Omega\times \bR^d}\int_0^1 dt\;\vf_\delta(x) \bar \eta_\delta(x-y)\langle\nabla g(x+(y-x)t),y-x\rangle\\
&=\int_{\Omega\times\bR^d}\int_0^1 dt\; t^{-d-1}\vf_\delta(x) \bar \eta_\delta(t^{-1}(x-\xi))\langle\nabla g(\xi),\xi-x\rangle\\
&=\delta\int_0^1 dt\int_{\bR^d}\div (\hat \eta_{t\delta} *\vf_\delta)\cdot g
\end{split}
\]
where $\hat \eta_{\delta}(x)=\bar\eta_\delta(x)x\delta^{-1}$. Hence $\|\chi_\delta(g-\bar \eta_\delta *g)\|_{L^1(\Omega)}\leq C_\eta\delta\|g\|_{BV(\bR^d)}$ and the same then holds for $g\in BV$.\footnote{Indeed, if $g\in BV$, then there exists $\{g_n\}\subset \cC^1$, $\sup_n\|g_n\|_{BV}<\infty$ such that $g_n$ converges to $g$ in $L^1$ (see \cite[5.2.2, Theorem 2]{EG}) and the claim follows since convolutions are $L^1$ contractions.} Accordingly, for each $h\in BV$,
\[
\|h-h_\delta\|_{L^1}\leq C_\eta\delta^{\frac\alpha d} \|h\|_{BV(\Omega)}.
\]
\end{proof}
Here is a useful consequence of the above Lemma: for each two point $x,y\in\bR^d$
\begin{equation}\label{eq:disto}
h_\delta(x)\leq e^{C_\eta\delta^{-d-1}\|x-y\|}h_\delta(y).
\end{equation}
This follows since, setting $g(t)=h_\delta(tx+(1-t)y)$, holds 
\[
\left|\ln \frac{h_\delta(x)}{h_\delta(y)}\right|=|\ln g(1)-\ln g(0)|\leq C_\#\|x-y\|\frac{\|\nabla h_\delta\|_{L^\infty}}{\inf h_\delta}\leq C_\eta \|x-y\|\delta^{-d-1}.
\]

\end{document}